\documentclass[12pt]{article}
\usepackage{graphics}
\usepackage{amsmath}
\usepackage{amsfonts}
\usepackage{enumerate}
\usepackage{latexsym}
\usepackage{epsfig}

\newtheorem{theorem}{Theorem}[section]
\newtheorem{lemma}[theorem]{Lemma}

\def\proofbox{\begin{picture}(6.5,6.5)
\put(0,0){\framebox(6.5,6.5){}}\end{picture}}
\newenvironment{proof}{\noindent{\it Proof.\quad}}{\hfill\proofbox}

\begin{document}

\title{Injective Simplicial Maps of the Complexes of Curves of Nonorientable Surfaces}
\author{Elmas Irmak}

\maketitle

\renewcommand{\sectionmark}[1]{\markright{\thesection. #1}}

\thispagestyle{empty}
\maketitle

\begin{abstract} Let $N$ be a compact, connected, nonorientable surface of genus $g$ with $n$ boundary components, and
$\mathcal{C}(N)$ be the complex of curves of $N$. Suppose that $g + n \leq 3$ or $g + n \geq 5$. If
$\lambda : \mathcal{C}(N) \rightarrow \mathcal{C}(N)$ is an injective simplicial map, then $\lambda$ is induced
by a homeomorphism of $N$.\end{abstract}

\maketitle

{\small Key words: Mapping class groups, simplicial maps, nonorientable surfaces

}

\section{Introduction}

The complex of curves of a compact, connected, orientable surface was defined as an abstract simplicial complex
by Harvey in \cite{Har} as follows: The vertex set consists of nontrivial simple closed curves, where nontrivial means it does not
bound a disk and it is not isotopic to a boundary component of the surface. Vertices form a simplex if they can be represented by
pairwise disjoint simple closed curves. For a compact, connected, nonorientable surface, $N$, the complex of curves, $\mathcal{C}(N)$,
is defined similarly where a simple closed curve is called nontrivial if it does not bound a disk,
a mobius band, and it is not isotopic to a boundary component of the surface. On nonorientable surfaces the author proved that superinjective
simplicial maps of complexes of curves are induced by homeomorphism in \cite{Ir5}, and improved that result by proving simplicial maps
that satisfy connectivity property are induced by homeomorphisms in \cite{Ir6}. In this paper we improve these results and prove that
injective simplicial maps are induced by homeomorphisms.\\

The main result of this paper is the following:

\begin{theorem} Let $N$ be a compact, connected, nonorientable surface of genus $g$ with
$n$ boundary components. Suppose that $g + n \leq 3$ or $g + n \geq 5$.
If $\lambda : \mathcal{C}(N) \rightarrow \mathcal{C}(N)$ is an injective simplicial map,
then $\lambda$ is induced by a homeomorphism $h : N \rightarrow N$ (i.e $\lambda([a]) = [h(a)]$
for every vertex $[a]$ in $\mathcal{C}(N)$).\end{theorem}

Simplicial maps of the complexes of curves are studied to get information about the mapping class groups. On orientable surfaces,
the extended mapping class group is defined as the group of isotopy classes of all self-homeomorphisms of the surface.

On compact, connected, orientable surfaces: Ivanov proved that automorphisms of the complexes of curves are induced by homeomorphisms \cite{Iv1}.
By using this result he classified isomorphisms between any two finite index subgroups of the extended mapping class groups \cite{Iv1}.
Korkmaz proved Ivanov's results for lower genus cases \cite{K1}, and Luo gave a proof for all cases \cite{L}.
Ivanov-McCarthy classified injective homomorphisms between mapping class groups \cite{IMc}.
The mapping class group is isomorphic to the the automorphism group of several complexes on orientable surfaces. These isomorphisms are given for  complexes such as the complex of pants decompositions (proven by Margalit \cite{M}),
the complex of nonseparating curves (proven by the author \cite{Ir3}), the complex
of separating curves (proven by Brendle-Margalit \cite{BM1}, and McCarthy-Vautaw \cite{MV}), the complex of Torelli geometry (proven by
Farb-Ivanov \cite{FIv}), the Hatcher-Thurston complex (proven by Irmak-Korkmaz \cite{IrK}), and the complex of arcs (proven by Irmak-McCarthy \cite{IrM}). Farb-Ivanov obtained applications showing that the automorphism group of the Torelli subgroup is isomorphic to the mapping class group \cite{FIv}. This result was extended by McCarthy-Vautaw to genus at least 3 \cite{MV}.

Superinjective simplicial maps were defined by the author in \cite{Ir1} on compact, connected orientable surfaces as simplicial maps of the complexes of curves which preserve geometric intersection zero and nonzero properties of the vertices. The author proved that superinjective simplicial
maps are induced by homeomorphisms of the surfaces. As an application she gave a classification of injective homomorphisms from
finite index subgroups of the extended mapping class group to the extended mapping class group for genus at least two \cite{Ir1}, \cite{Ir2}, \cite{Ir3}. Behrstock-Margalit and Bell-Margalit proved author's results for small genus cases \cite{BhM}, \cite{BeM}. Brendle-Margalit
proved that if $K$ is the subgroup of mapping class group generated by Dehn twists about separating curves, then any injection from a finite index subgroup of $K$ to the Torelli group is induced by a homeomorphism \cite{BM1}. They obtained this as an application, after proving that
superinjective simplicial maps of separating curve complex are induced by homeomorphisms. Shackleton proved that injective simplicial maps also behave this way, i.e. they are also induced by homeomorphisms, and he obtained strong local co-Hopfian results as applications \cite{Sh}.

Kida proved several results about superinjective simplicial maps on orientable surfaces in \cite{Ki1}, \cite{Ki2}, \cite{Ki3},
and as applications he proved that for all but finitely many compact orientable surfaces the abstract commensurators of
the Torelli group and the Johnson kernel for such surfaces are naturally isomorphic to the extended mapping class group,
any injective homomorphism from a finite index subgroup of the Johnson kernel into the Torelli group for such a surface is
induced by an element of the extended mapping class group, any finite index subgroup of the Johnson kernel is
co-Hopfian. Irmak-Ivanov-McCarthy proved
that each automorphism of a surface braid group is induced by a homeomorphism of the underlying
surface, provided that this surface is a closed, connected, orientable surface of genus at least 2, and the number of strings is
at least three in \cite{IIM}. Kida and Yamagata also proved several results about superinjective simplicial maps in \cite{KiY1}, \cite{KiY2},
\cite{KiY3}, and as applications they gave a description of any injective homomorphism from a finite index subgroup of the pure braid
group with $n$ strands on a closed orientable surface of genus $g$ into the pure braid group. They proved that the abstract
commensurator of the braid group with $n$ strands on a closed orientable surface
of genus $g$ is naturally isomorphic to the extended mapping class group of a
compact orientable surface of genus $g$ with n boundary components. They also proved that for a connected,
compact and orientable surface of genus two with one boundary component any finite index subgroup
of the Torelli group for S is co-Hopfian.

Mapping class groups and abstract simplicial complexes on nonorientable surfaces are not studied as much as the orientable case.
Here are some known results for nonorientable surfaces: Atalan proved that the automorphism group of the curve complex
is isomorphic to the mapping class group for most odd genus cases. The author proved that each injective simplicial map of the
complex of arcs is induced by a homeomorphism, and the automorphism group of the complex of arcs is isomorphic to the mapping class
group in most cases \cite{Ir4}. Atalan-Korkmaz proved that the automorphism group of the curve complex is isomorphic to the
mapping class group for most cases \cite{AK}. They also proved that two curve complexes are isomorphic if and only if the
two surfaces they are defined on are homeomorphic. The author proved that each superinjective simplicial map of the complex of curves
is induced by a homeomorphism in most cases \cite{Ir5}. She also proved that if a simplicial map of the curve complex, satisfies the
connectivity property, i.e. ``two vertices are connected by an edge if and only if their images are connected by an edge'', then it is
induced by a homeomorphism \cite{Ir6}. This result implies that superinjective simplicial maps and automorphisms are induced by homeomorphisms.
Our main result in this paper improves the above results about simplicial maps of the complex of curves. The case when $g + n = 4$ is open.

\section{Injective Simplicial Maps}

In this section we will assume that $N$ is a compact, connected, nonorientable surface of genus $g$ with $n$ boundary components.
We will work on pair of pants decompositions on $N$. We define them as follows: If $a$ is a simple close curve on $N$, let $N_a$ be
the cut surface along $a$. A set, $P$, of pairwise disjoint, nonisotopic, nontrivial simple closed curves on $N$ is called a pair of
pants decomposition of $N$ if each component of $N_P$ is a pair of pants. The set of isotopy classes of elements of $P$ forms the vertices
of a maximal simplex of $\mathcal{C}(N)$. Every maximal simplex of $\mathcal{C}(N)$ is obtained by the isotopy classes of elements of some pair of
pants decomposition of $N$. All maximal simplices in the complexes of curves have the same dimension on orientable surfaces, but $\mathcal{C}(N)$
has different dimensional maximal simplices. We show several pants decompositions on a closed surface of genus 7 in Figure \ref{New-Fig0}.
These pants decompositions correspond to maximal simplices in $\mathcal{C}(N)$. They have different dimensions. There are cross signs in
the figures. This means that we remove the interiors of the disks which have cross signs in them, and then we identify the antipodal points
of the resulting boundary components.\\

\begin{figure}
\begin{center}
\epsfxsize=2.2in \epsfbox{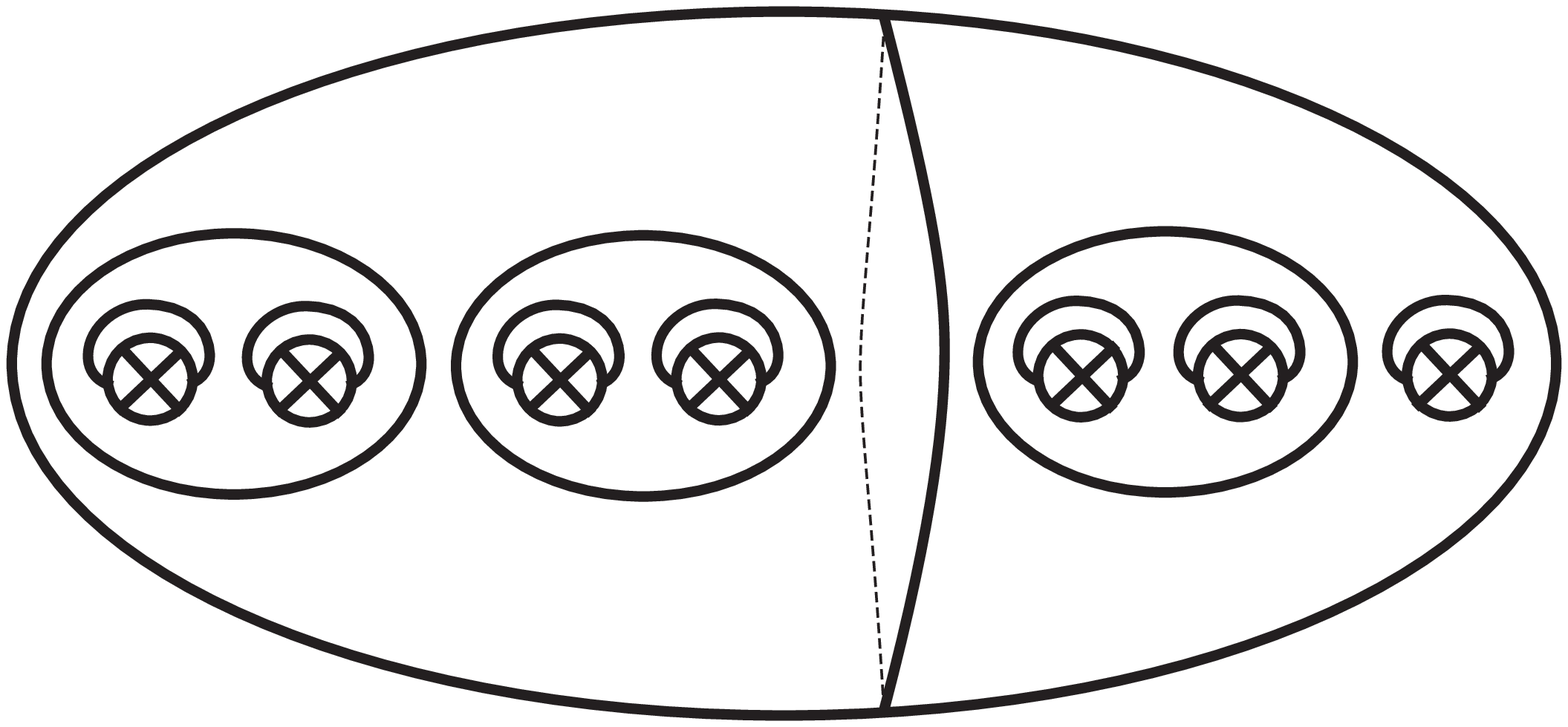} \hspace{0.1in} \epsfxsize=2.2in
\epsfbox{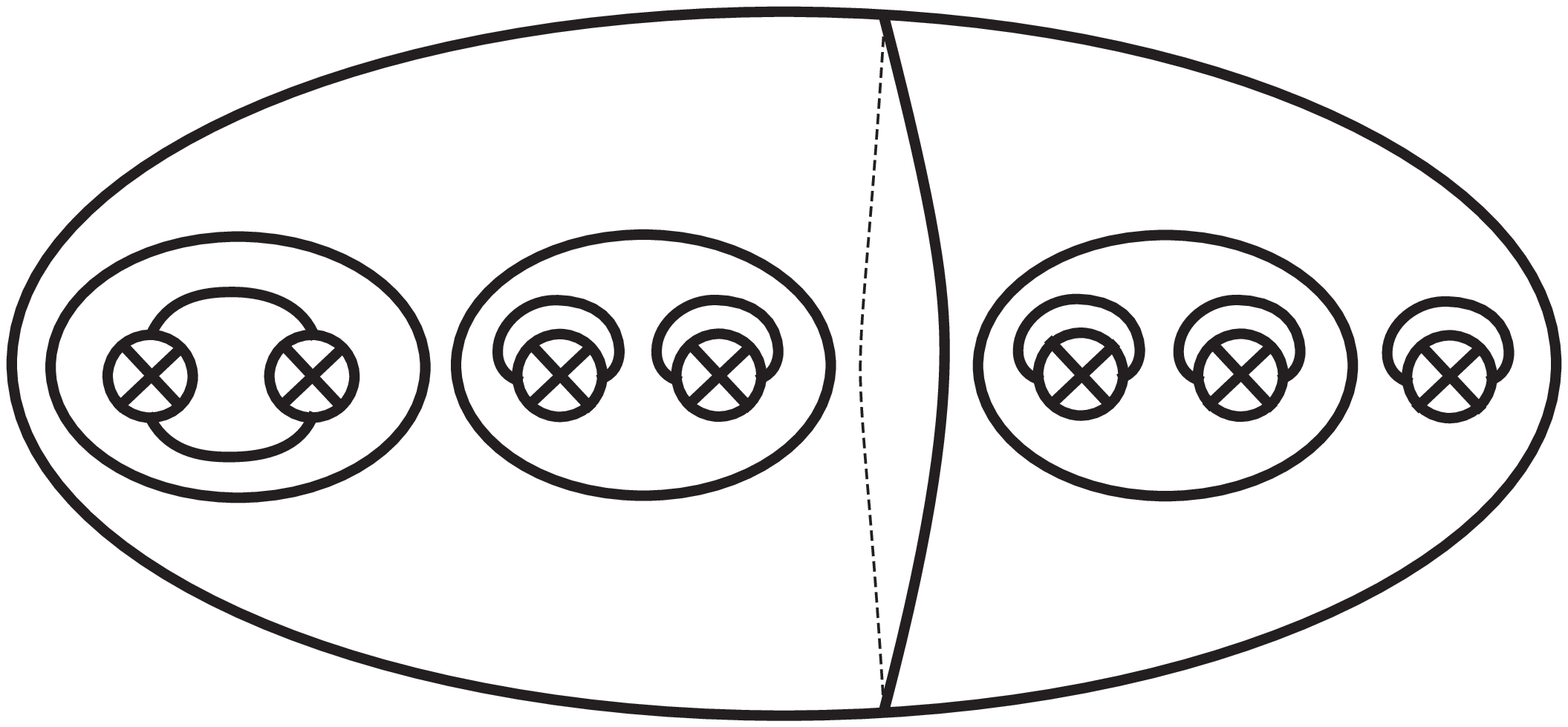} \hspace{0.1in} \epsfxsize=2.2in \epsfbox{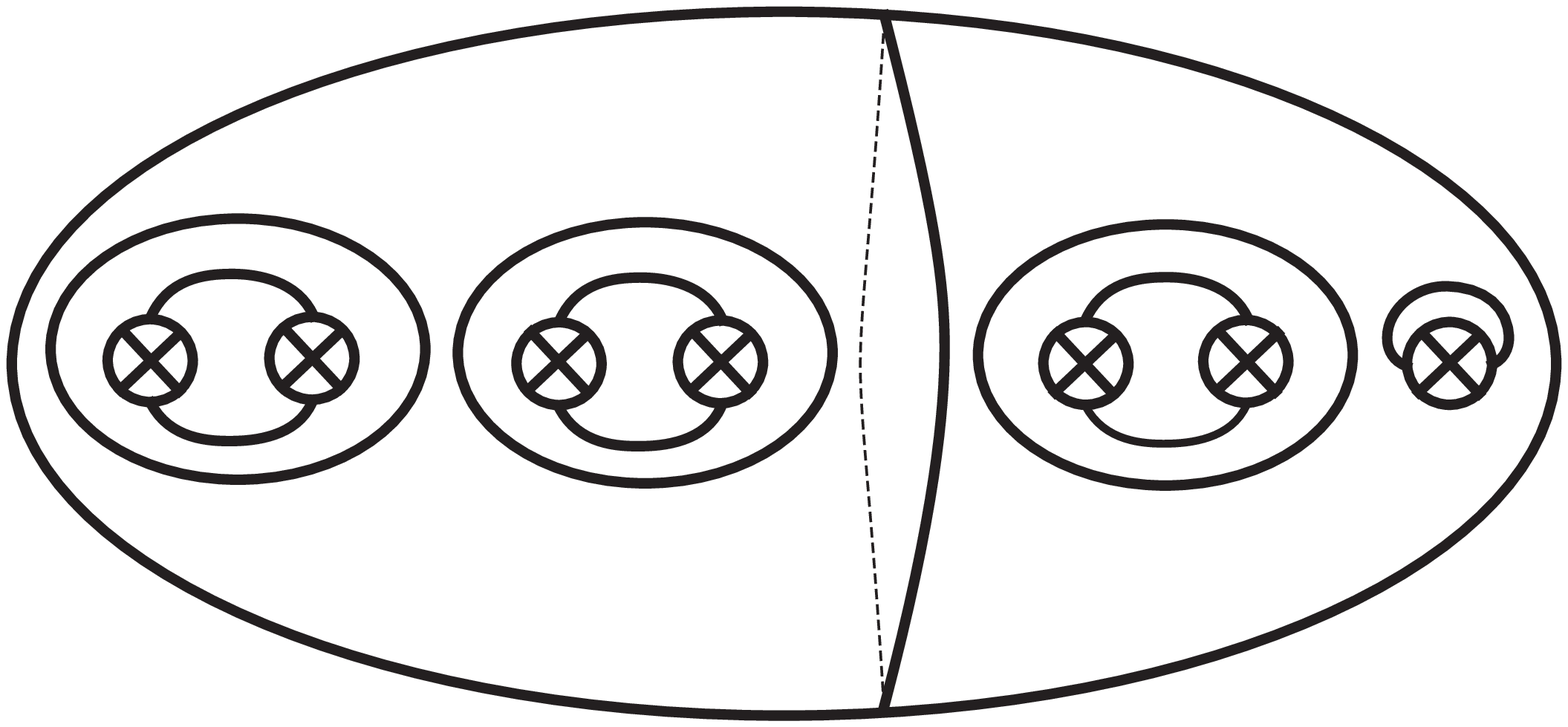}
\caption{Pants decompositions}
\label{New-Fig0}
\end{center}
\end{figure}

The following lemma is given in \cite{A} and \cite{AK}:

\begin{lemma}
\label{dim} Let $N$ be a nonorientable surface of genus $g \geq 2$ with $n$ boundary components. Suppose that $(g, n) \neq (2, 0)$.
Let $a_r= 3r+n-2$ and $b_r = 4r +n -2$ if $g = 2r+1$, and let $a_r= 3r+n-4$ and $b_r= 4r +n -4$ if $g =2r$.
Then there is a maximal simplex of dimension $q$ in $\mathcal{C}(N)$ if and only if $a_r \leq q \leq b_r$.\end{lemma}

Let $\Delta$ be a maximal simplex in $\mathcal{C}(N)$. $\Delta$ will be called a top dimensional maximal simplex if $\Delta$ has
the highest dimension.

\begin{lemma}
\label{tp} Let $g \geq 2$. Suppose that $(g, n) = (3, 0)$ or $g+n \geq 4$. Let $P$ be pair of pants which corresponds to
a top dimensional maximal simplex in $\mathcal{C}(N)$. The curves in $P$ are either separating or 1-sided whose complement is nonorientable.
In $P$, the number of 1-sided curves whose complement is nonorientable is $g$, and the number of separating curves
is  $2r + n - 2$ if $g = 2r+1$, and $2r +n -3$ if $g=2r$.
\end{lemma}

\begin{proof} Let $P$ be a pair of pants decomposition such that it corresponds to a top dimensional maximal simplex, $\Delta$,
in $\mathcal{C}(N)$. Let $a \in P$. If $a$ is a 1-sided simple closed curve whose complement is orientable, then the genus
of $N$ is odd, say $g = 2r +1$ for some $r \in \mathbb{Z}$. We see that $r \geq 1$ as $g \geq 2$. The complement of $a$ is an
orientable surface of genus $r$ with $n+1$ boundary components. On an orientable surface of genus $g_0$ with $n_0$ boundary
components, all maximal simplices have dimension $3g_0 + n_0 -4$. This implies that $[a]$ can be a vertex of at most a
$3r + n - 2$ dimensional simplex in $\mathcal{C}(N)$. Since the dimension of $\Delta$ is $4r + n - 2$ and $r \geq 1$, we get
a contradiction. So, $a$ is not a 1-sided simple closed curve whose complement is orientable.

Suppose that $a$ is a 2-sided nonseparating simple closed curve. Suppose $g=2r$ for some $r \in \mathbb{Z}$. Then $r \geq 1$ as
$g \geq 2$. If the complement of $a$ is nonorientable, then it has genus $2r-2$
and $n+2$ boundary components. Then by Lemma \ref{dim}, there is at most a $4r + n - 5$ dimensional simplex containing $[a]$ as a
vertex in $\mathcal{C}(N)$. Since the dimension of $\Delta$ is $4r + n -4$, we get a contradiction. If the complement of $a$ is
orientable, then the complement is an orientable surface of genus $r-1$ with $n + 2$ boundary components. In this case there is
at most a $3r + n - 4$ dimensional simplex containing $[a]$ in $\mathcal{C}(N)$. Since the dimension of $\Delta$ is
$4r + n -4$ and $r \geq 1$, we get a contradiction. Suppose $g =2r + 1 $ for some $r \in \mathbb{Z}$. In this case complement
of $a$ is a nonorientable surface of genus $2r-1$ with $n+2$ boundary components. By Lemma \ref{dim}, there is at most a
$4r + n - 3$ dimensional simplex containing $[a]$ in $\mathcal{C}(N)$. Since the dimension of $\Delta$ is $4r + n - 2$,
we get a contradiction. So, $a$ is not a 2-sided nonseparating simple closed curve on $N$.

Hence, the curves in $P$ are either 1-sided curves whose complements are nonorientable or separating curves. Gluing pair of pants 
along curves that come from cutting along separating curves on the surface, doesn't give a nonorientable surface. So, in $P$ the 
number of 1-sided curves with nonorientable complements is $g$. By using Lemma \ref{dim}, we see that the number of separating 
curves is $2r + n - 2$ if $g = 2r+1$, and $2r +n -3$ if $g=2r$.\end{proof}\\

Let $a, b$ be two nonisotopic nontrivial simple closed curves such that they have $i([a], [b]) \neq 0$, nonzero geometric 
intersection. We will say that $a$ and $b$ have small intersection if there exists a pair of pants decomposition
$P$ on $N$ which corresponds to a top dimensional maximal simplex in $\mathcal{C}(N)$ such that $a \in P$ and
$(P \setminus \{a\}) \cup \{b\}$ also corresponds to a top dimensional maximal simplex in $\mathcal{C}(N)$.

\begin{lemma}
\label{smallint} Suppose that $g + n \geq 4$. Let $\lambda : \mathcal{C}(N) \rightarrow \mathcal{C}(N)$ be an injective simplicial map.
If $a$ and $b$ have small intersection, then $i(\lambda([a]), \lambda([b])) \neq 0$. \end{lemma}

\begin{proof} Suppose $a$ and $b$ have small intersection. We complete $a$ to a pair of pants decomposition $P$ on $N$ which corresponds
to a top dimensional maximal simplex in $\mathcal{C}(N)$ such that $(P \setminus \{a\}) \cup \{b\}$ also corresponds to a top dimensional maximal simplex in $\mathcal{C}(N)$. Let $P'$ be a set of pairwise disjoint curves representing $\lambda([P])$. Since $\lambda$ is injective,
$P'$ corresponds to a top dimensional maximal simplex. Since $b$ doesn't intersect any of the curves in $P \setminus \{a\}$, $i(\lambda([b], \lambda([x])) = 0$ for any $x \in P \setminus \{a\}$. Since $\lambda([b])$ is not isotopic to $\lambda([y])$ for any $y \in P$, and $P'$
corresponds to a top dimensional maximal simplex, we see that $i(\lambda([a]), \lambda([b])) \neq 0$.\end{proof}\\

\begin{figure}
\begin{center}
\epsfxsize=2.5in \epsfbox{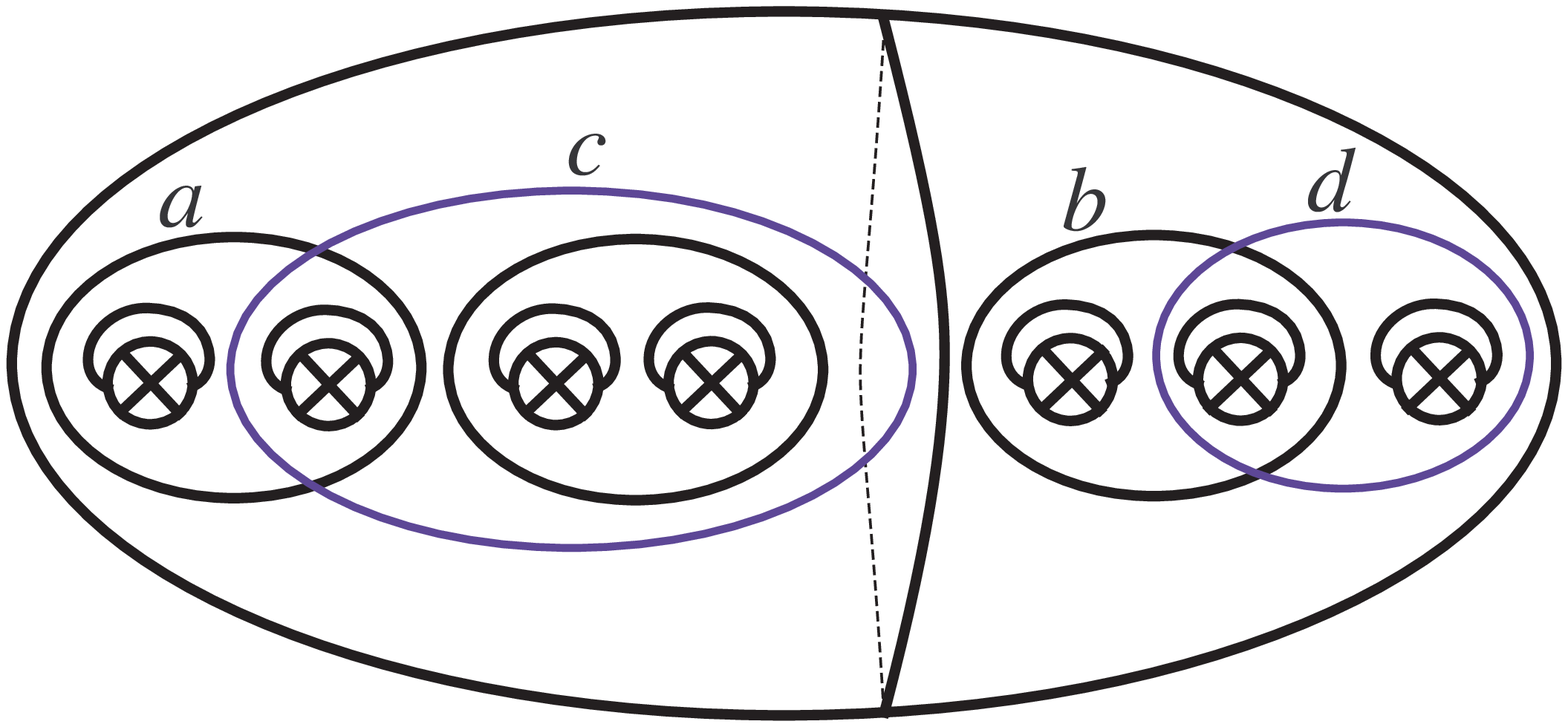} \hspace{0.2in} \epsfxsize=2.5in
\epsfbox{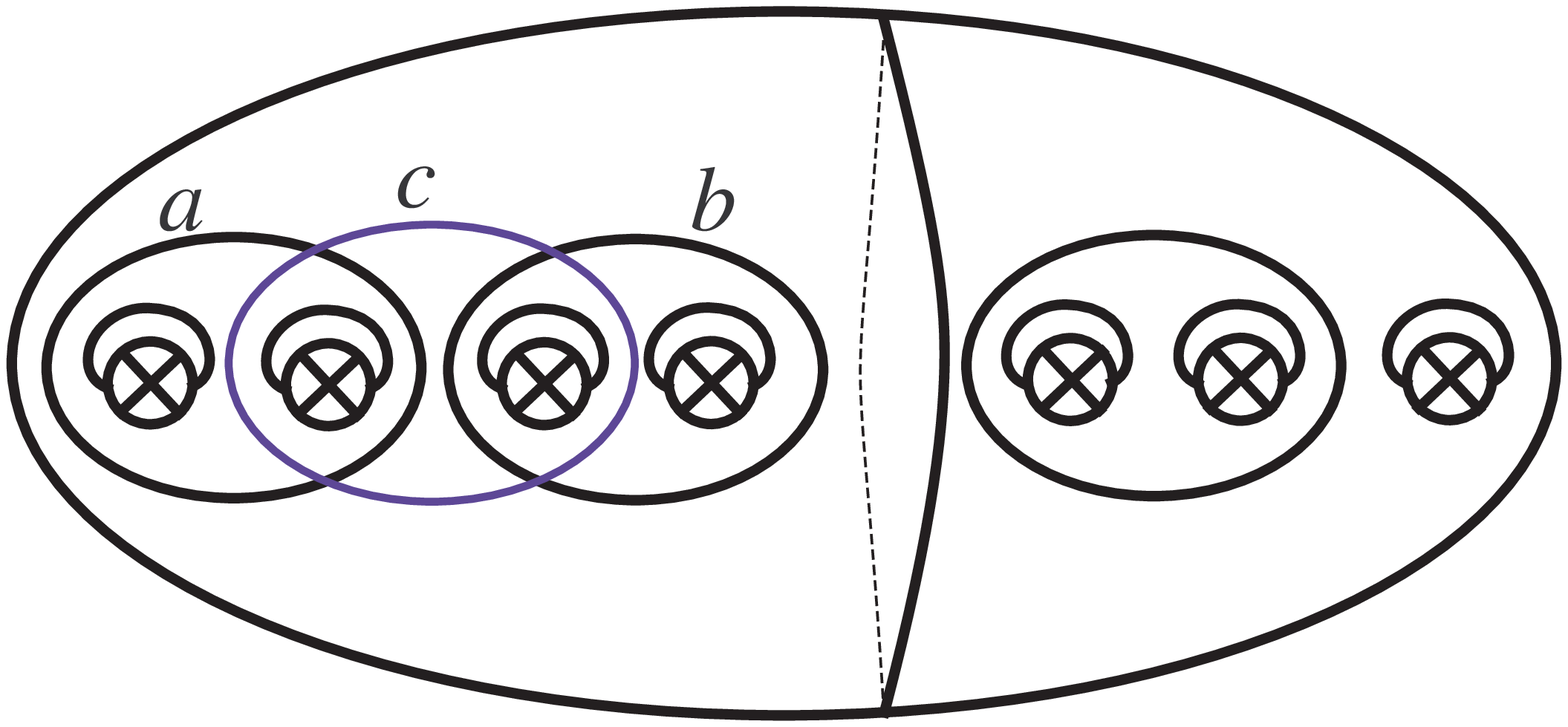}

\hspace{0.5cm} (i) \hspace{6.5cm}  (ii) \vspace{0.5cm}

\epsfxsize=2.5in \epsfbox{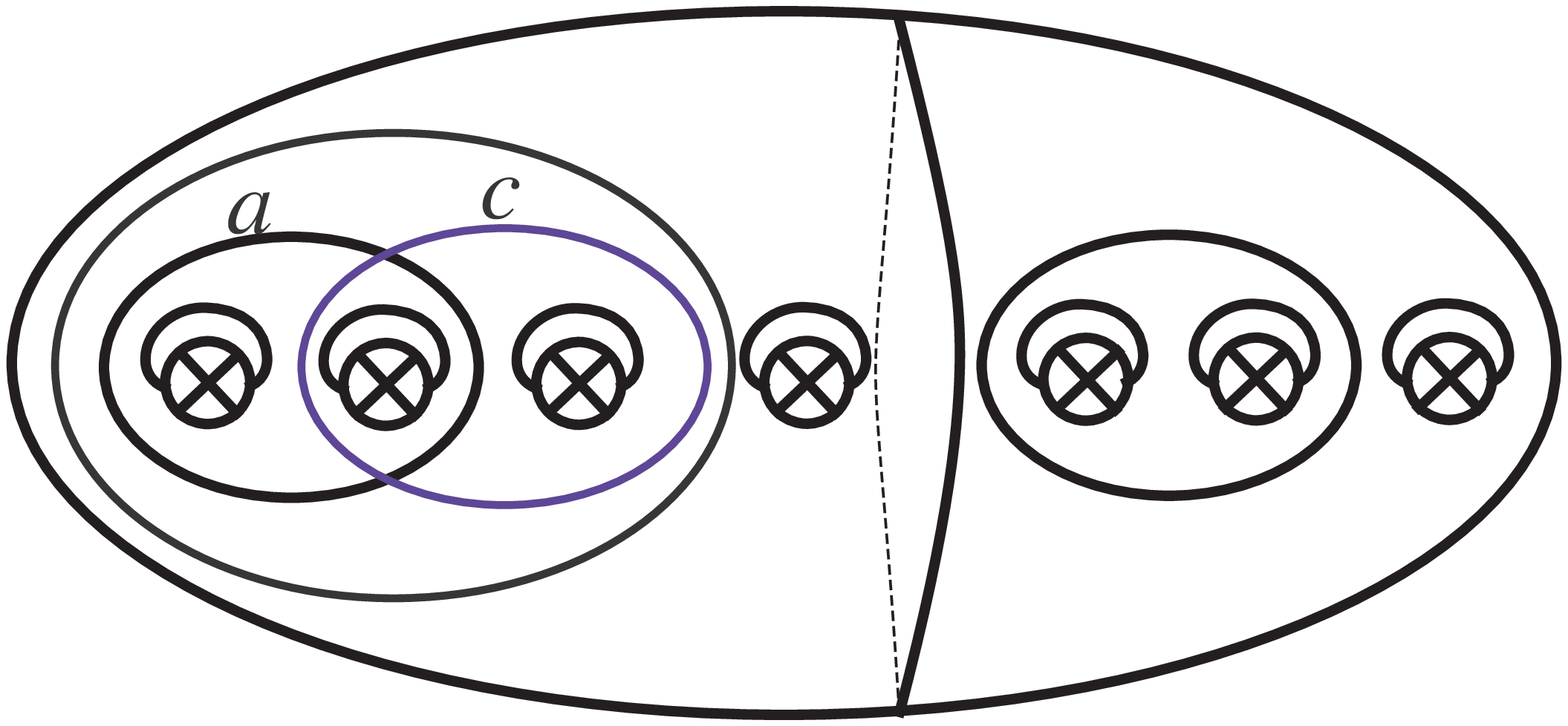} \hspace{0.2in} \epsfxsize=2.5in
\epsfbox{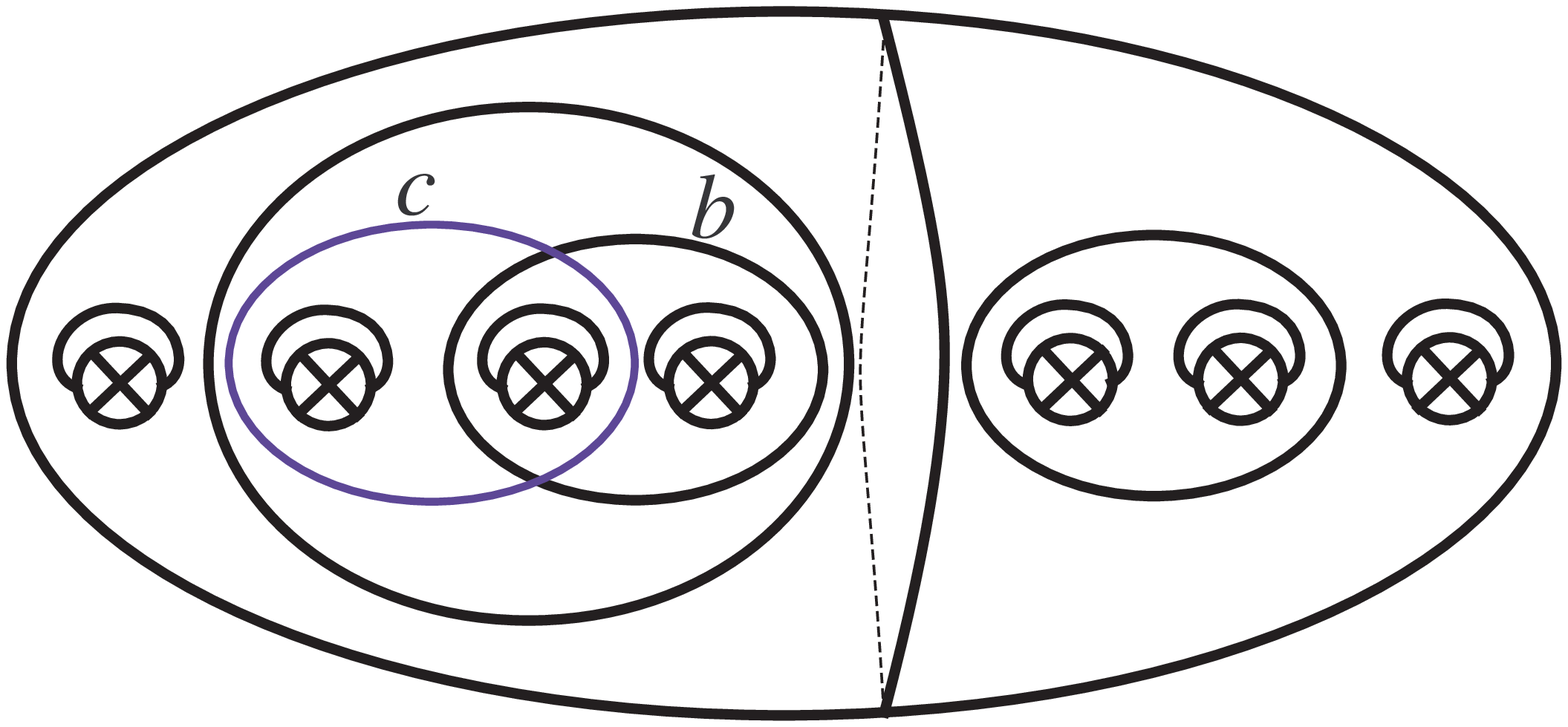}

\hspace{0.5cm} (iii) \hspace{6.5cm}  (iv)
\caption{Curves on closed surface of genus 7}
\label{Figure-adj}
\end{center}
\end{figure}

Let $a$ and $b$ be two distinct elements in a pair of pants decomposition $P$ on $N$ where $P$ corresponds to a top dimensional maximal simplex in $\mathcal{C}(N)$. Then $a$ is called adjacent to $b$ w.r.t. $P$ iff there exists a pair of pants in $P$ which has $a$ and $b$ on
its boundary. In the following lemmas we will see that adjacency and nonadjacency are preserved w.r.t. top dimensional maximal simplices.

\begin{lemma}
\label{nonadjacent} Suppose that $g + n \geq 4$. Let $\lambda : \mathcal{C}(N) \rightarrow \mathcal{C}(N)$ be an injective
simplicial map. Let $P$ be a pair of pants decomposition on $N$ which corresponds to a top
dimensional maximal simplex in $\mathcal{C}(N)$. Let $a, b \in P$ such that $a$ is not adjacent to $b$
w.r.t. $P$. There exists $a'  \in \lambda([a])$ and $b'  \in \lambda([b])$ such that $a'$ is not adjacent
to $b'$ w.r.t. $P'$ where $P'$ is a set of pairwise disjoint curves representing $\lambda([P])$ containing $a', b'$.\end{lemma}

\begin{proof} Let $a$ and $b$ be two curves in $P$. Assume that $a$ is not adjacent to $b$ w.r.t. $P$. We can choose simple closed curves
$c, d$ on $N$ such that $c$ and $a$ have small intersection, $c$ doesn't intersect any of the curves in $P \setminus \{a\}$, 
$d$ and $b$ have small intersection, $d$ doesn't intersect any of the curves in $P \setminus \{b\}$, $c$ and $d$ are disjoint,
and $a, b, c, d$ are all pairwise nonisotopic. In Figure \ref{Figure-adj} (i), we show how to choose $c, d$ for a special case on
a closed surface of genus 7 case, where $a, b, P$ are as shown in the figure. Let $a'  \in \lambda([a])$ and $b'  \in \lambda([b])$.
We choose a set of pairwise disjoint simple closed curves representing $\lambda([P])$ containing $a', b'$, call it $P'$. By using 
Lemma \ref{smallint} we see that $i(\lambda([a]), \lambda([c])) \neq 0$, $i(\lambda([b]), \lambda([d])) \neq 0$. Since $\lambda$ is 
injective we also  know that none of the two elements in $\{\lambda([a]), \lambda([b]), \lambda([c]), \lambda([d]) \}$ are equal. Since 
$c$ doesn't intersect any of the curves in $P \setminus \{a\}$, we see that $\lambda([c])$ has geometric intersection zero with any element in 
$\lambda([P]) \setminus \{ \lambda([a]) \}$. Similarly, $\lambda([d])$ has geometric intersection zero with any element in
$\lambda([P]) \setminus \{ \lambda([b]) \}$. Since $i([c], [d]) = 0$, we see that $i(\lambda([c]), \lambda([d])) = 0$.  This is possible 
only when $a'$ is not adjacent to $b'$ w.r.t. $P'$.
\end{proof}

\begin{lemma}
\label{1-sd} Suppose that $(g, n) =(1, 4)$ or $(2, 2)$. Let $\lambda : \mathcal{C}(N) \rightarrow \mathcal{C}(N)$ be an
injective simplicial map. If $a$ is a 1-sided simple closed curve on $N$, then there exists $a' \in \lambda([a])$ such that
$a'$ is a 1-sided simple closed curve on $N$.
\end{lemma}

\begin{proof} Let $a$ be a 1-sided simple closed curve. If $(g, n)= (1, 4)$, we complete $a$ to a pants decomposition
$P= \{a, z, t\}$ as shown in Figure \ref{Fig01} (i). Let $a' \in \lambda([a]), z' \in \lambda([z]), t' \in \lambda([t])$
and $a', z', t'$ have minimal intersection. Let $P' = \{a', z', t'\}$. $P'$ corresponds a top dimensional maximal simplex in $\mathcal{C}(N)$.
Since $a$ is not adjacent to $t$, and nonadjacency is preserved by Lemma \ref{nonadjacent}, we see that $z'$ has
to be a separating curve, $a'$ and $t'$ have to be on different sides of $z'$. Suppose $a'$ is a separating curve. Then
$a', z', t'$ are as shown in Figure \ref{Fig01} (ii). Then, we get a contradiction by using the curves $x, y$
given in part (i). The curves $x, y, a, z, t$ are pairwise nonisotopic and each $x$ and $y$ has small intersection with $t$,
and disjoint from $z$. Let $x', y'$ be representatives of $\lambda([x])$ and $\lambda([y])$ which have minimal intersection
with each of $a', z', t'$. Since $\lambda$ is injective, $x', y', a', z', t'$ are pairwise nonisotopic. By using Lemma \ref{smallint},
we see that $x'$ and $y'$ should intersect $t'$ and should be disjoint from $z'$. This gives a contradiction. So, $a'$ is not
a separating curve. Since $g=1$, $a'$ is a 1-sided curve.

If $(g, n)= (2, 2)$, we complete $a$ to a pants decomposition $P= \{a, y, b\}$ as shown in Figure \ref{Fig01} (iii).
Let $a' \in \lambda([a]), y' \in \lambda([y]), b' \in \lambda([b])$ and $a', y', b'$ have minimal intersection. Let
$P' = \{a', y', b'\}$. $P'$ corresponds to a top dimensional maximal simplex. By Lemma \ref{tp}, we know that $a'$ is either
a separating curve or a 1-sided curve. Since $a$ is not adjacent to $b$ w.r.t. $P$, and nonadjacency is preserved by
Lemma \ref{nonadjacent}, we see that $y'$ has to be a separating curve, and $a'$ and $b'$ have to be on different sides of
$y'$. This implies that $a', y', b'$ are as shown in Figure \ref{Fig01} (iv). Hence, $a'$ is a 1-sided curve.\end{proof}

\begin{figure}
\begin{center}
\epsfxsize=2.1in \epsfbox{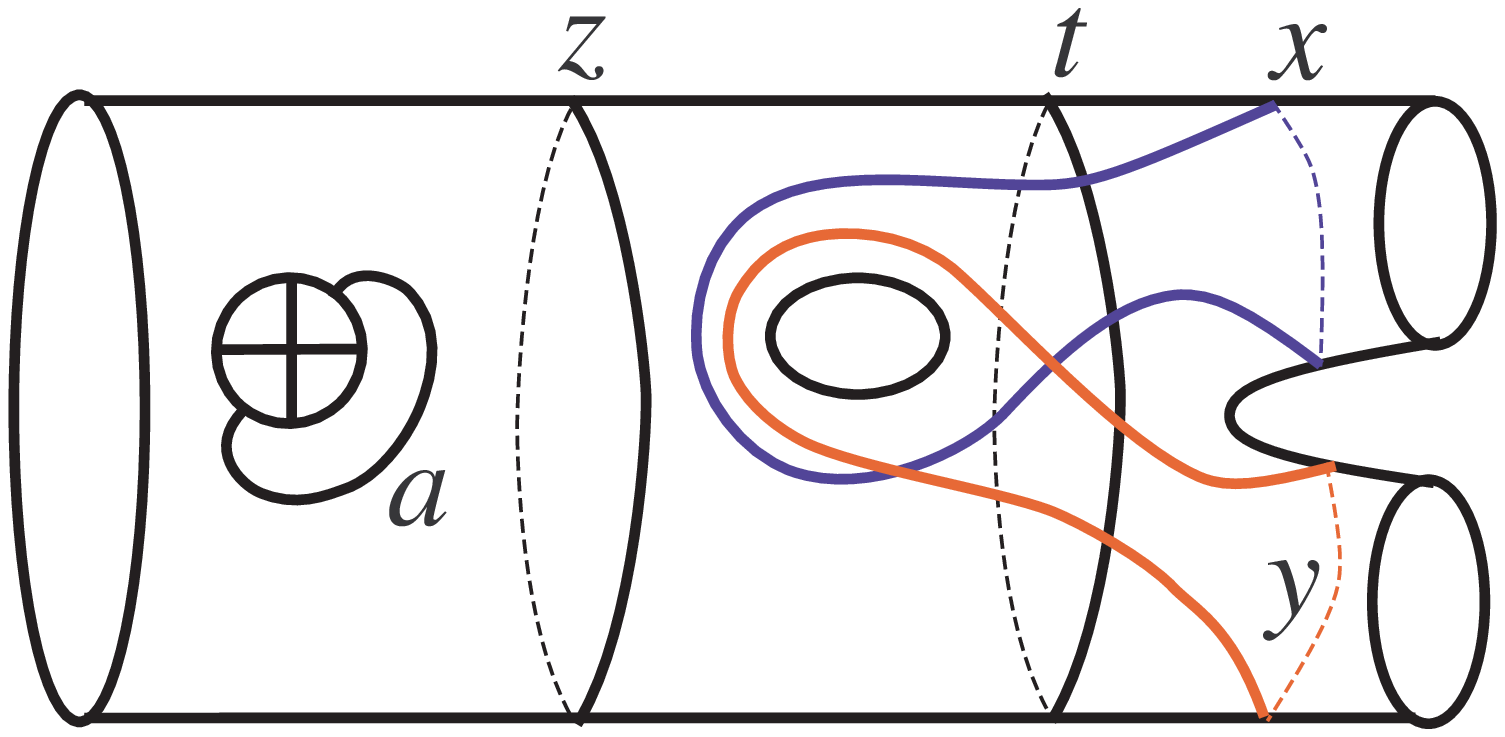}  \hspace{0.5cm} \epsfxsize=2.1in \epsfbox{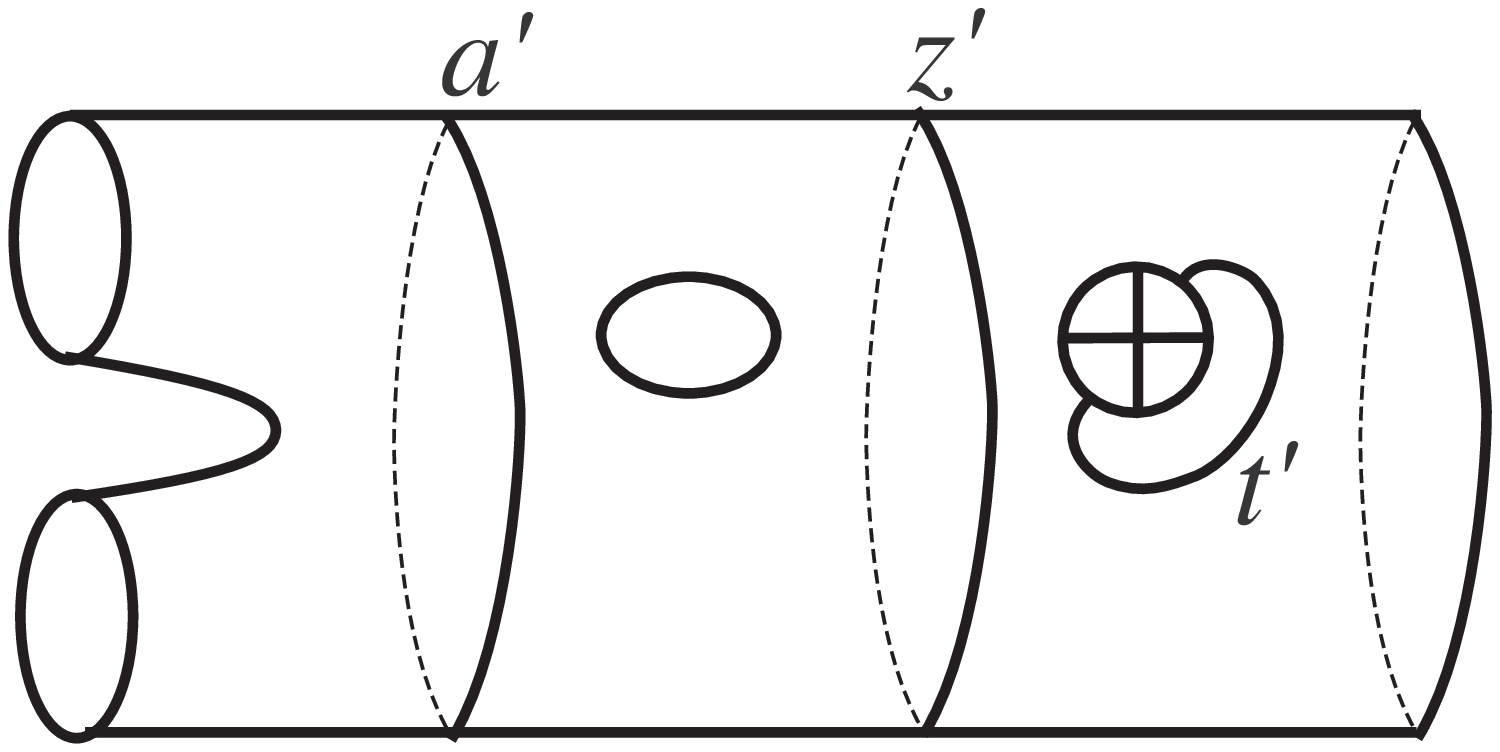}

(i) \hspace{5cm} (ii) \vspace{0.1in}

\epsfxsize=1.78in \epsfbox{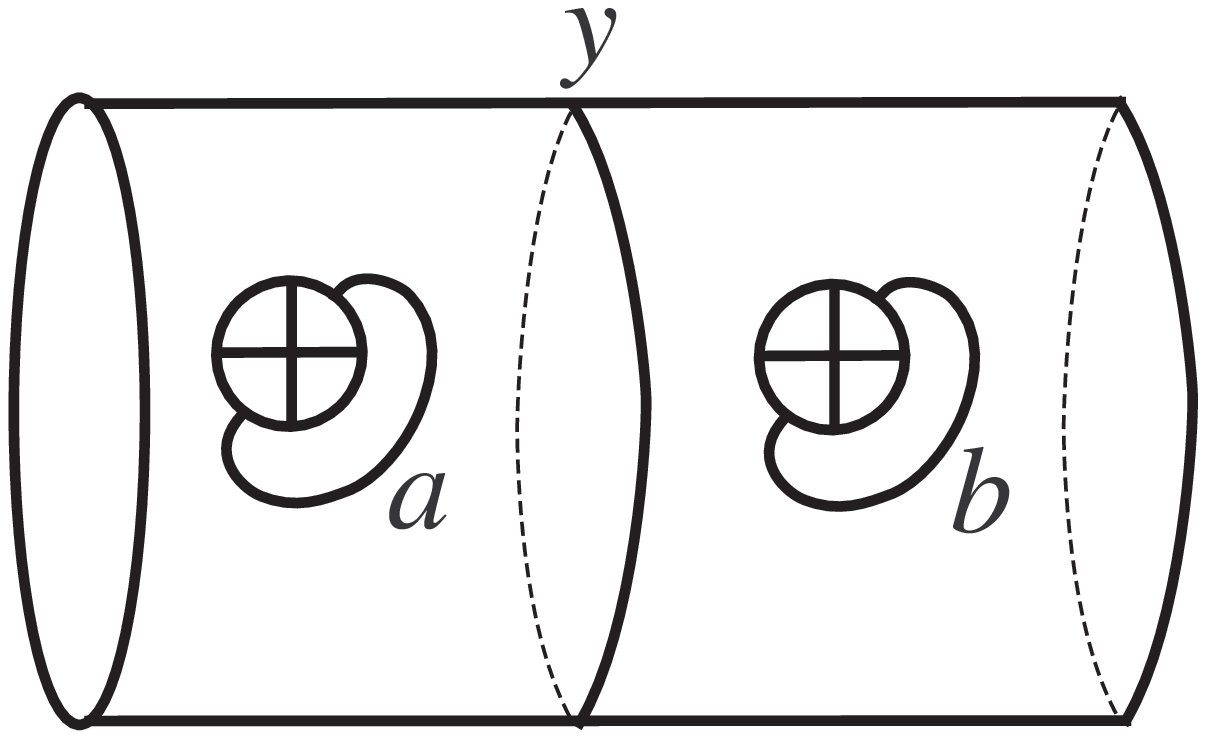}  \hspace{0.5cm} \epsfxsize=1.78in \epsfbox{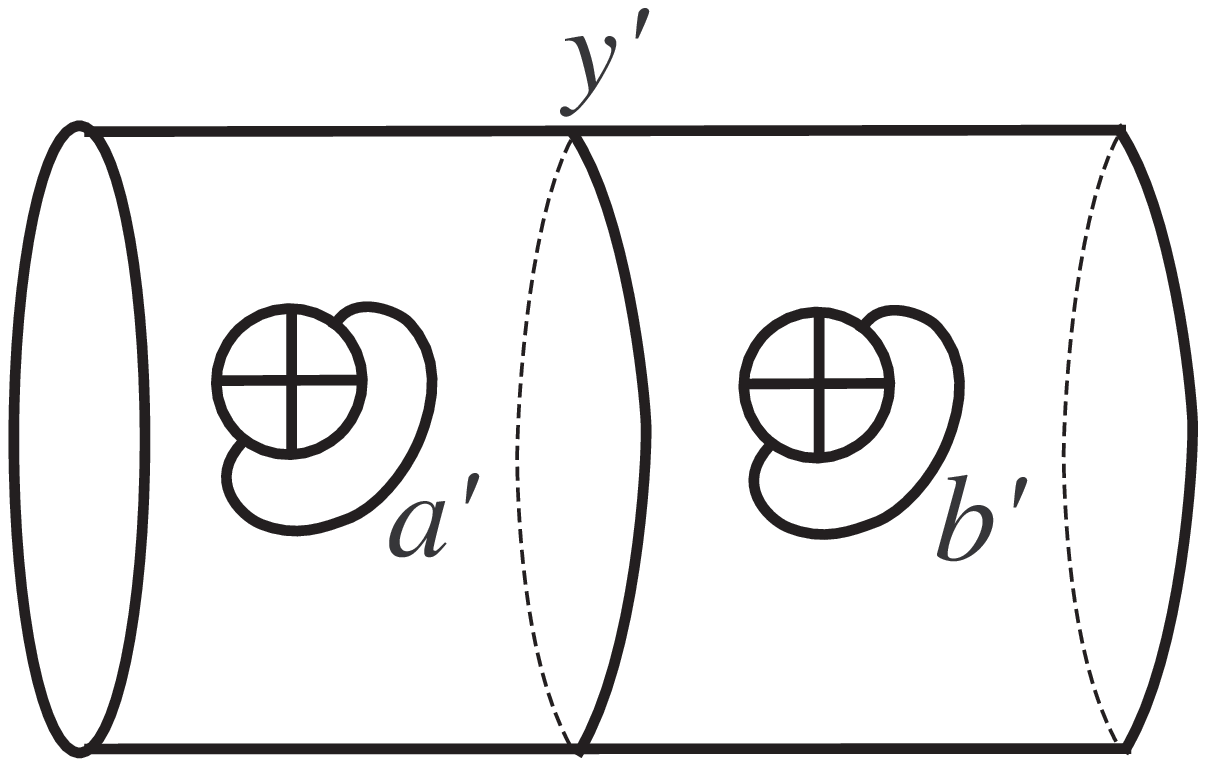}

(iii) \hspace{4.45cm} (iv)
\caption{Curve configurations I}
\label{Fig01}
\end{center}
\end{figure}

\begin{lemma}
\label{adjacent1} Suppose that $g + n \geq 5$. Let $\lambda : \mathcal{C}(N) \rightarrow \mathcal{C}(N)$ be an injective
simplicial map. Let $P$ be a pair of pants decomposition on $N$ which corresponds to a top dimensional maximal simplex in
$\mathcal{C}(N)$. Let $a, b \in P$ such that $a$ and $b$ are both 2-sided and $a$ is adjacent to $b$ w.r.t. $P$. There
exists $a' \in \lambda([a])$ and $b' \in \lambda([b])$ such that $a'$ is adjacent to $b'$ w.r.t. $P'$ where $P'$ is a
set of pairwise disjoint curves representing $\lambda([P])$ containing $a', b'$.
\end{lemma}

\begin{proof} Let $a, b \in P$ such that $a$ and $b$ are both 2-sided and $a$ is adjacent to $b$ w.r.t. $P$. We can choose
a simple closed curve $c$ on $N$ such that $a$ and $c$ have small intersection, $b$ and $c$ have small intersection,
and $c$ doesn't intersect any other curve in $P$. In Figure \ref{Figure-adj} (ii) we show how to choose the curve $c$ for
a special case on a closed surface of genus 7, where $a, b, P$ are shown in the figure. In Figure \ref{Figure-adj} (iii)
we show why $a$ and $c$ have small intersection. In Figure \ref{Figure-adj} (iv) we show why $b$ and $c$ have small intersection.
By using Lemma \ref{smallint}, we see that $\lambda([a])$ and $\lambda([c])$ have nonzero geometric intersection, and $\lambda([b])$
and $\lambda([c])$ have nonzero geometric intersection. We also see that $\lambda([c])$ has zero geometric intersection with all
the other elements in $\lambda([P])$. This implies that there exists $a' \in \lambda([a])$ and $b' \in \lambda([b])$ such that $a'$
is adjacent to $b'$ w.r.t. $P'$ where $P'$ is a set of pairwise disjoint curves representing $\lambda([P])$ containing $a', b'$.
\end{proof}\\

To show that adjacency is preserved for other types of simple closed curves we give the following lemmas.

\begin{lemma}
\label{adjacent2} Let $g \geq 2$. Suppose that $(g, n)= (3, 0)$ or $g + n \geq 4$. Let
$\lambda : \mathcal{C}(N) \rightarrow \mathcal{C}(N)$ be an injective simplicial map. Let $P$ be a pair of pants
decomposition on $N$ which corresponds to a top dimensional maximal simplex in $\mathcal{C}(N)$. Let $a, b \in P$
such that $a$ and $b$ are both 1-sided curves and $a$ is adjacent to $b$ w.r.t. $P$. There exists $a'  \in \lambda([a])$
and $b'  \in \lambda([b])$ such that $a'$ is adjacent to $b'$ w.r.t. $P'$ where $P'$ is a set of pairwise disjoint
curves representing $\lambda([P])$ containing $a', b'$.
\end{lemma}

\begin{proof} Suppose that $(g, n)= (3, 0)$ or $g + n \geq 4$. Let $a, b \in P$ such that $a$ and $b$ are both 1-sided,
and $a$ is adjacent to $b$ w.r.t. $P$. Let $P'$ be a set of pairwise disjoint curves representing $\lambda([P])$.

The statement is easy to see if $(g, n)= (3, 0)$, as there are only three 1-sided curves in $P$. Since $\lambda$ is injective,
there are three curves in $P'$. They are all 1-sided and adjacent to each other w.r.t. $P'$.

Assume that $(g, n) \neq (3, 0)$. There exists $x \in P$ such that $a, b, x$ is as
in Figure \ref{Fig100} (i). Let $a' \in \lambda([a])$, $b' \in \lambda([b])$,
$x' \in \lambda([x])$ such that $a', b', x' \in P'$.

Suppose $g + n \geq 4$ and $(g, n) \neq (2, 2)$. There is at least one other curve in $P$ which is
on the other side of $x$, and it is adjacent to $x$ w.r.t. $P$.  Since $x$ is a separating curve which has curves that are
adjacent to it on both sides, and nonadjacency is preserved by Lemma \ref{nonadjacent}, $x'$ can't be a 1-sided curve,
otherwise some curves that are not adjacent w.r.t. $P$ would have to be adjacent w.r.t. $P'$, and this would give a contradiction.
By Lemma \ref{tp}, the only curves in a top dimensional simplex are either 1-sided or separating curves. So, since $x'$ is not 1-sided,
$x'$ is a separating curve. Since $a, b$ are not adjacent to any other curve then $x$, we see that $a', b'$ are not
adjacent to any other curve then $x'$. This implies that $a', b'$ has to be on the same side of $x'$, and there shouldn't
be any other curve coming from $P'$ on that side. This implies that $a'$ and $b'$ have to be adjacent w.r.t. $P'$.

Suppose $(g, n) = (2, 2)$, then by Lemma \ref{1-sd}, we know that both $a'$ and $b'$ are 1-sided. Then, $x'$ is a separating curve. Suppose
$a'$ and $b'$ are not adjacent w.r.t. $P'$. Then, $a', x', b'$ are as shown in Figure \ref{Fig100} (ii). This gives a contradiction as $\lambda$ is injective, and there are infinitely many nonisotopic 1-sided simple closed curves that are disjoint from $x$ on $N$, so their images should have nonisotopic representatives disjoint from $x'$, but there are only finitely many nonisotopic simple closed curves on $N$ disjoint from $x'$. Hence, $a'$ and $b'$ have to be adjacent w.r.t. $P'$.\end{proof}

\begin{figure}
\begin{center}
\epsfxsize=1.7in \epsfbox{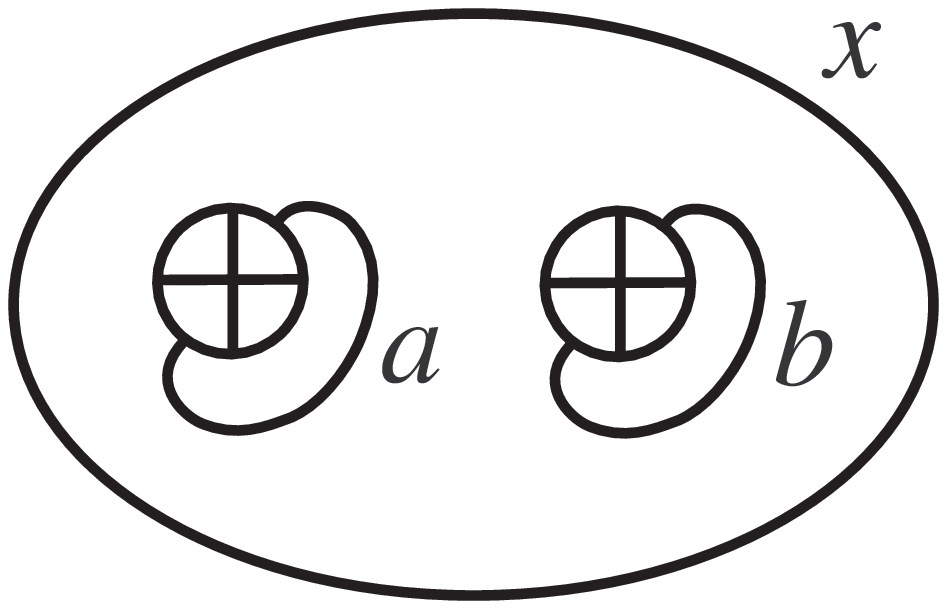} \hspace{0.1in} \epsfxsize=1.8in \epsfbox{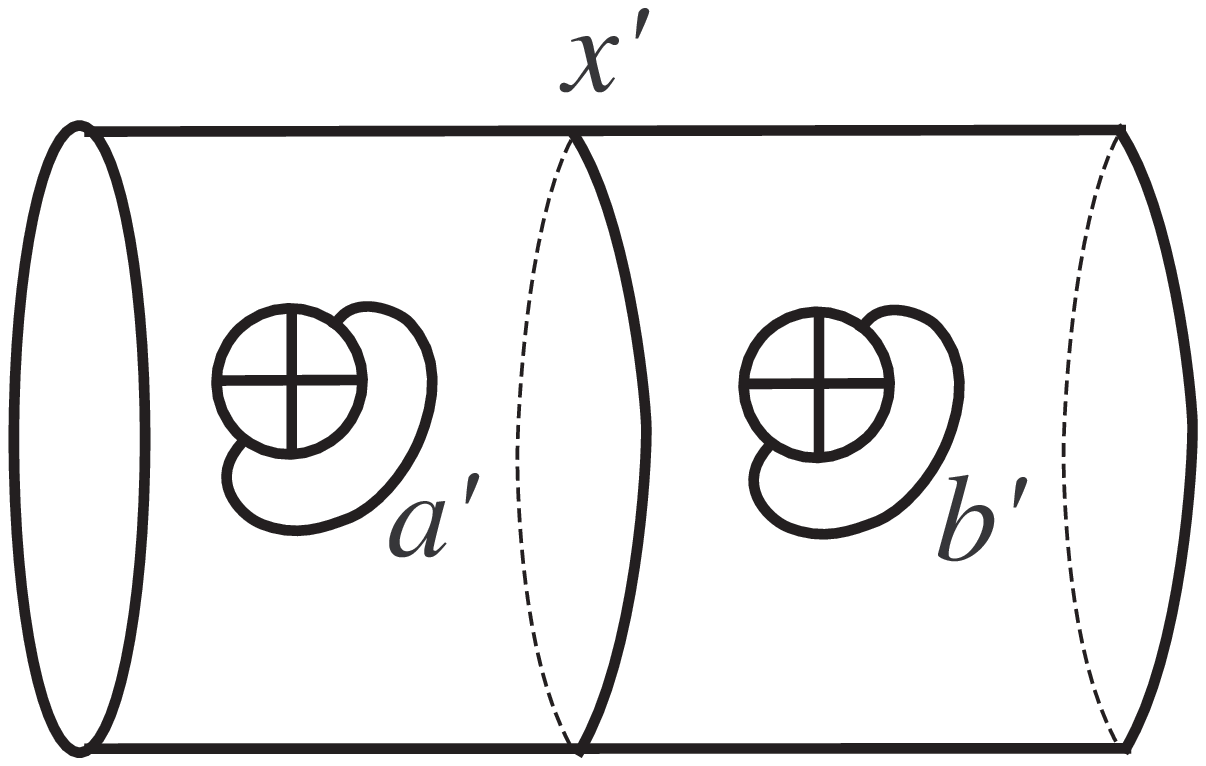} \hspace{0.1in}
\epsfxsize=1.95in \epsfbox{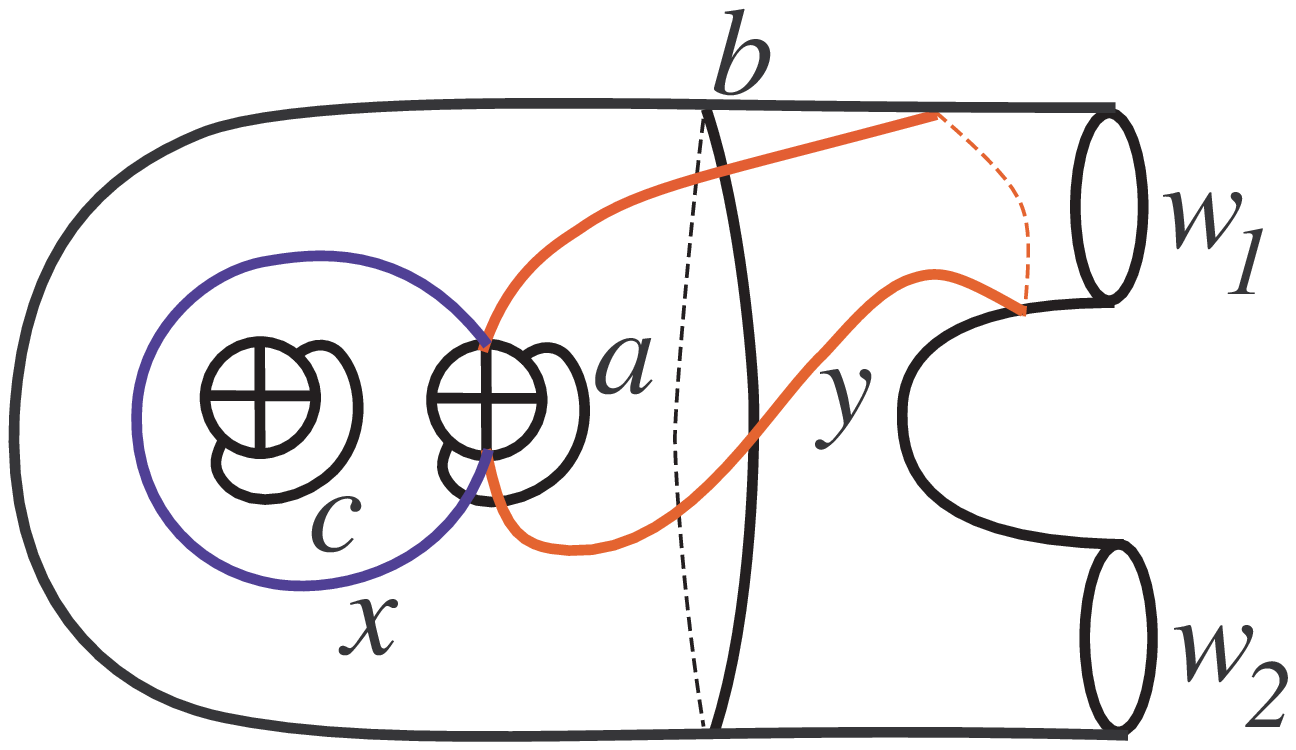}

 \hspace{0.3cm} (i) \hspace{3.92cm} (ii) \hspace{4.6cm} (iii)  \vspace{0.2in}

\epsfxsize=2.632in \epsfbox{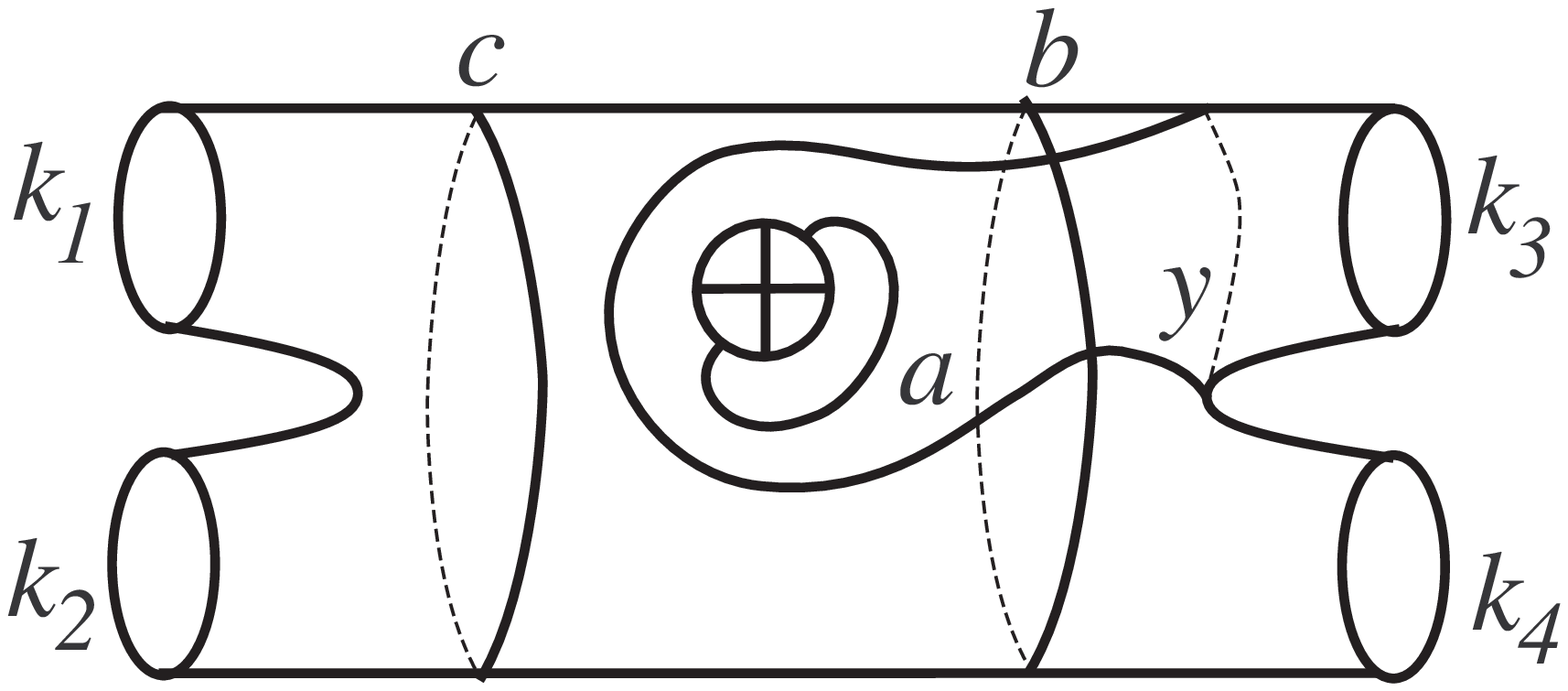} \hspace{0.3in}
\epsfxsize=2.3in \epsfbox{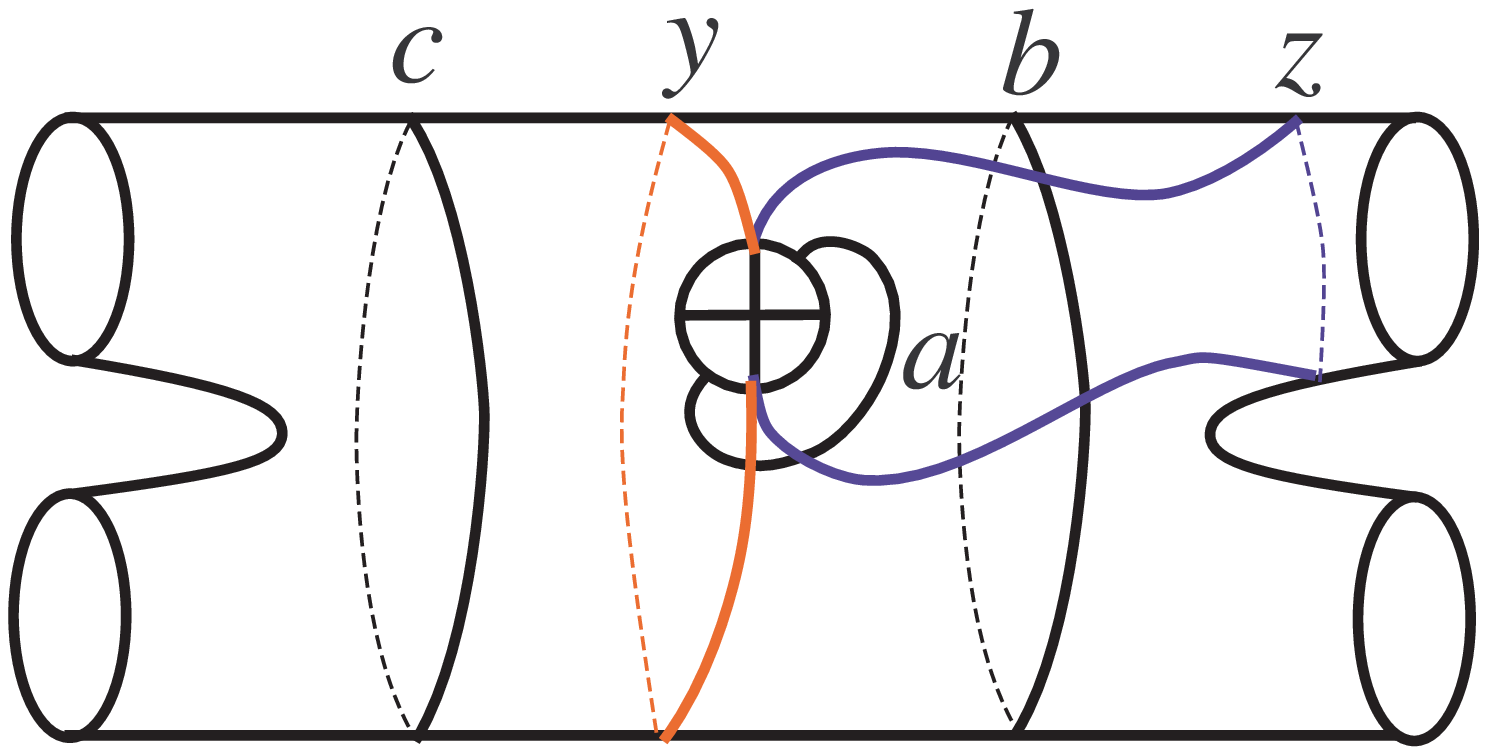}

\hspace{0.38cm} (iv) \hspace{6.56cm} (v) \vspace{0.2in}

\caption{Curve configurations II}
\label{Fig100}
\end{center}
\end{figure}

\begin{lemma}
\label{adjacent3} Suppose that $g + n \geq 4$. Let $\lambda : \mathcal{C}(N) \rightarrow \mathcal{C}(N)$ be an
injective simplicial map. Let $P$ be a pair of pants decomposition on $N$ which corresponds to a top dimensional maximal
simplex in $\mathcal{C}(N)$. Let $a, b \in P$ such that $a$ is 1-sided, $b$ is 2-sided and $a$ is adjacent to $b$ w.r.t.
$P$. There exists $a'  \in \lambda([a])$ and $b'  \in \lambda([b])$ such that $a'$ is adjacent
to $b'$ w.r.t. $P'$ where $P'$ is a set of pairwise disjoint curves representing $\lambda([P])$ containing $a', b'$.
\end{lemma}

\begin{proof} Suppose that $g + n \geq 4$. Let $P$ be a pair of pants decomposition on $N$
which corresponds to a top dimensional maximal simplex in $\mathcal{C}(N)$. Let $a, b \in P$ such that
$a$ is 1-sided, $b$ is 2-sided, and $a$ is adjacent to $b$ w.r.t. $P$. Let $P'$ be a set of pairwise disjoint
curves representing $\lambda([P])$. Let $a' \in \lambda([a]), b' \in \lambda([b])$ such that $a', b' \in P'$.
The statement is easy to see in $(g, n) = (1,3)$ case as there are only two curves in $P$ if $(g, n) = (1,3)$.

Assume that $(g, n) \neq (1,3)$. By Lemma \ref{tp}, $b$ is a separating curve. Suppose $a$ is the only curve that is adjacent
to $b$ on one side of $b$. Since $g + n \geq 4$ and $(g, n) \neq (1, 3)$ there is at least one other curve in $P$ that is on
the other side of $b$. Since nonadjacency is preserved, we see that $b'$ should be a separating curve, and $a'$ has to be on
one side of $a'$ and there shouldn't be any other curve coming from $P'$ on that side. This implies that $a'$ is adjacent to $b'$. In the
other cases, by using Lemma \ref{tp}, we see that there exists $c \in P$ such that $a, b, c$ are as shown in Figure \ref{Fig100} (iii) or (iv).
Let $c' \in \lambda([c])$ such that $a', b', c'$ have minimal intersection. We note that the curves $w_1, w_2, k_1, k_2, k_3, k_4$
that we see in these figures could be representing boundary components of $N$ or separating
curves in $P$ or they could bound Mobius bands depending on the cases we will consider below.

\begin{figure}
\begin{center}
\epsfxsize=1.8in \epsfbox{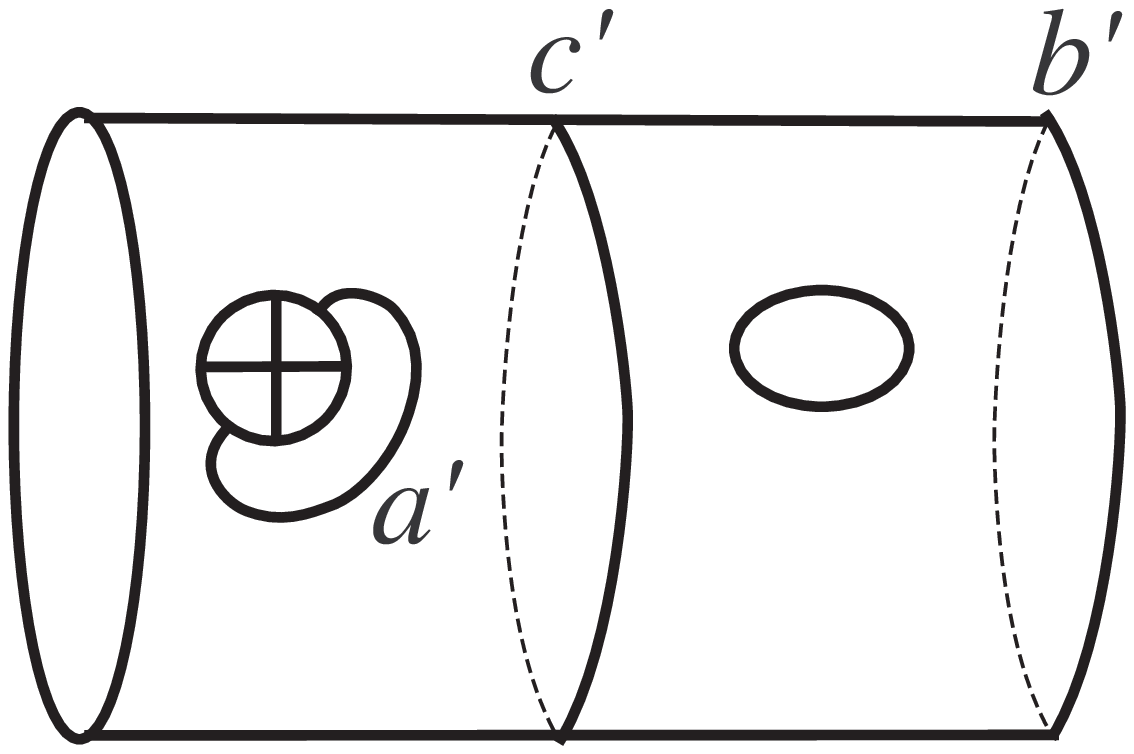} \hspace{0.3in}
\epsfxsize=2.3in \epsfbox{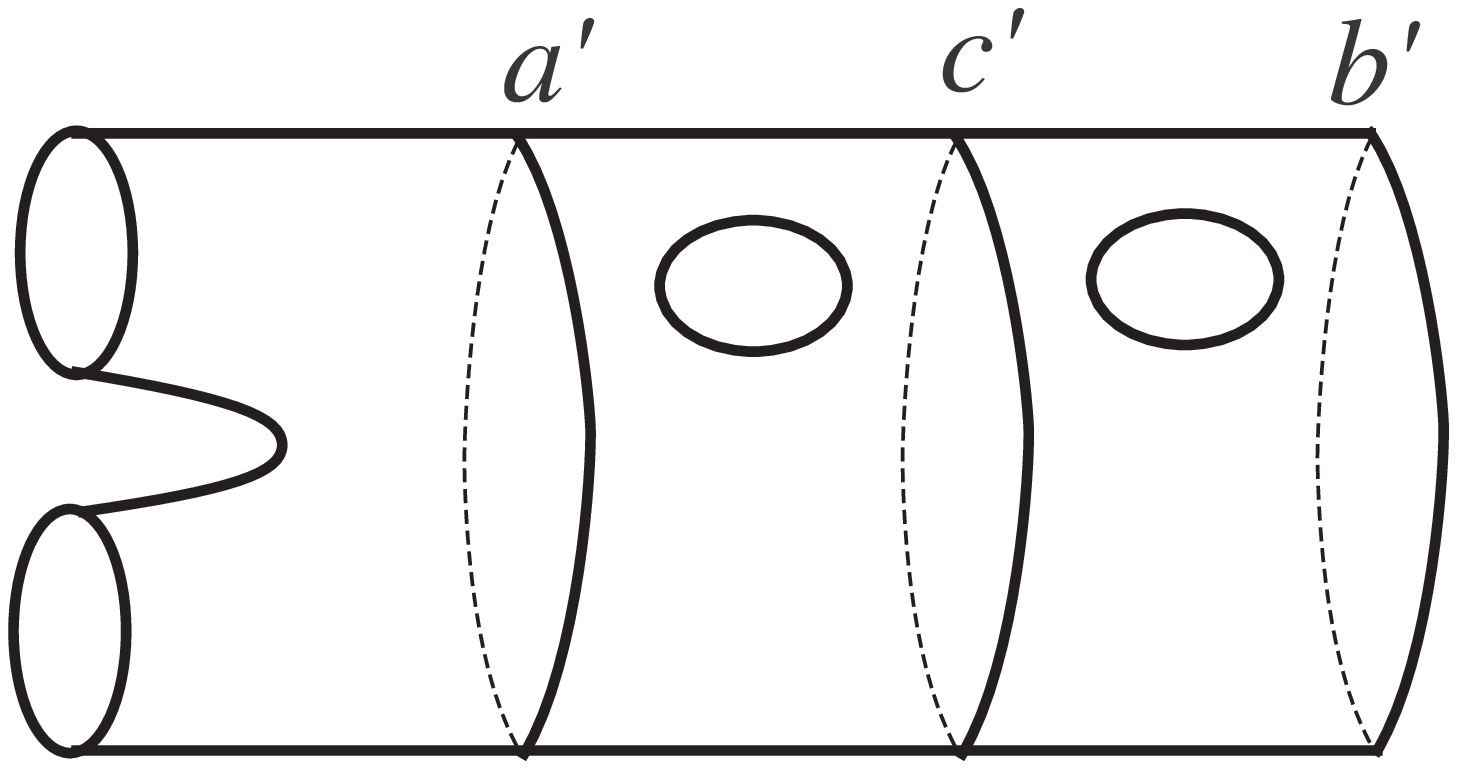}

\hspace{-0.35cm} (i) \hspace{5.6cm} (ii) \vspace{0.2in}

\epsfxsize=3.7in \epsfbox{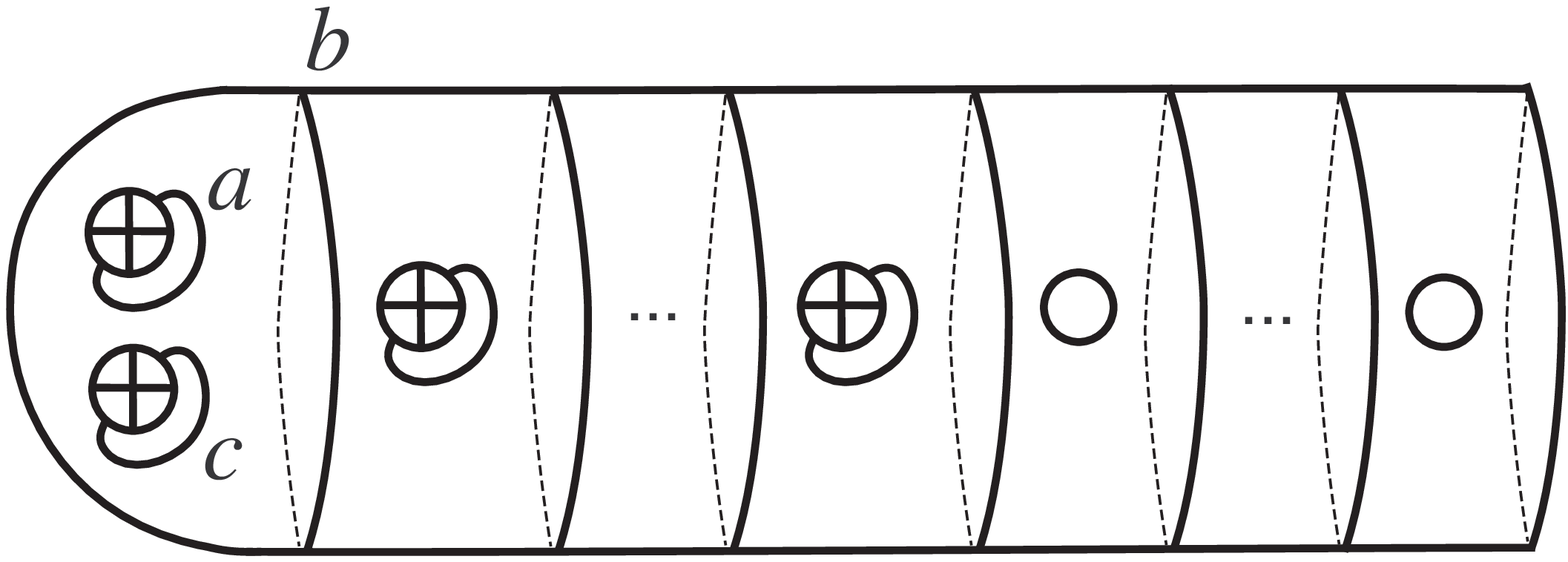}

(iii)
\caption{Curve configurations III}
\label{Fig101}
\end{center}
\end{figure}

{\bf Case 1:} Suppose $a, b, c$ are as shown in Figure \ref{Fig100} (iii). If $(g, n) = (2, 2)$, then by Lemma \ref{1-sd} and the previous
part, we see that $a', c'$ have to be both 1-sided and adjacent to each other w.r.t. $P'$. This implies that $a', b'$ also have to
be adjacent to each other w.r.t. $P'$. Suppose $(g, n) \neq (2, 2)$.
Then $b$ has curves that are adjacent to it on both sides. Since nonadjacency is preserved,
$b'$ is has to be a separating curve, and $a'$ and $c'$ has to be on the same side of $b'$ and there shouldn't be any other curve on that side coming from $P'$. Suppose $a'$ is not adjacent to $b'$. Then, $c'$ has to be adjacent to $b'$ w.r.t. $P'$. This implies that $a', b', c'$ are as shown in Figure \ref{Fig101} (i) or (ii). If $a', b', c'$ are as shown in Figure \ref{Fig101} (i),
we get a contradiction by using the curves $x, y$ shown in Figure \ref{Fig100} (iii). Because $x, y, a$ are pairwise nonisotopic,
and $x, y$ have small intersection with $a$, and they are both disjoint from $c$. Let $x' \in \lambda([x])$ and $y' \in \lambda([y])$ such that $a', x', y', c'$ have minimal intersection. Since $\lambda$ is injective, $a', x', y', c'$ are pairwise nonisotopic.
By Lemma \ref{smallint}, each of $x'$ and $y'$ has to intersect $a'$ essentially and should be disjoint from $c'$,
that gives a contradiction. Suppose $a', b', c'$ are as shown in Figure \ref{Fig101} (ii). We can complete
$a, b, c$ to a top dimensional pants decomposition $W$ (see Figure \ref{Fig101} (iii)), such that there are $g$ 1-sided curves, and the rest are separating curves and there exists at most one separating curve, $v$, in $W$, such that $v$ is not adjacent to any curve in $W$ on one side of it.
So, except possibly for one separating curve, separating curves in $W$ are adjacent to at least one curve on both sides w.r.t. $W$. Since nonadjacency is preserved, the image of all the separating curves in $W$ except possibly the image of $v$ should have separating representatives. This implies that the image of $g$ 1-sided curves in $W$ has to have at least $g-1$ 1-sided curves as representatives, by Lemma \ref{tp}. Hence, in our case in Figure \ref{Fig101} (ii), we get a contradiction because both $a'$ and $c'$ can't be separating curves at the same time, as $a$ and $c$ are 1-sided curves.
Hence, $a'$ has to be adjacent to $b'$.

\begin{figure}
\begin{center}
\epsfxsize=1.8in \epsfbox{Figure-623a.eps} \hspace{0.3in}
\epsfxsize=2.3in \epsfbox{Figure-56.eps}

\hspace{-0.4cm} (i) \hspace{5.6cm} (ii) \vspace{0.2in}

\epsfxsize=2.3in \epsfbox{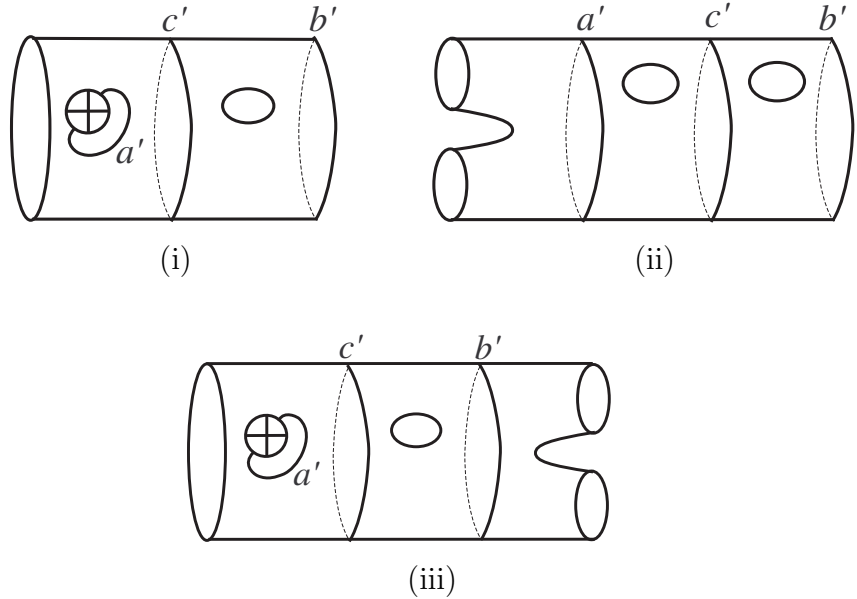} \hspace{0.3in}

(iii)
\caption{Curve configurations IV}
\label{Fig102}
\end{center}
\end{figure}

{\bf Case 2:} Suppose $a, b, c$ are as shown in Figure \ref{Fig100} (iv). Suppose there is a curve, $x \in P$, on the side of $b$ which doesn't
contain $a$. By using Lemma \ref{tp}, we see that $b'$ is a separating curve as nonadjacency is preserved and $b$ has curves that are adjacent to it on both sides, and $a'$ and $c'$ are on the same side of $b'$. Since $a$ is not adjacent to any other curve then $b$ and $c$, we see that $a'$ will not be adjacent to any other curve then $b'$ and
$c'$. If there is a third curve of $P$ which is on the same side of $b$ as $a$ and $c$, that will imply that $a'$ has to be adjacent to $b'$ since nonadjacency is preserved. Suppose
there is no other curve of $P$ on the side of $b$ which contains $a$ and $c$. We will see that $a'$ is adjacent to $b'$ w.r.t $P'$ as follows:
Suppose $a'$ is not adjacent to $b'$ w.r.t. $P'$. Then $c'$ has to be adjacent to $b'$ and $a'$, and $a', b', c'$ have to be as shown in Figure \ref{Fig102} (i) or \ref{Fig102} (ii). In both cases we get a contradiction as follows: By changing the curve $b$ to $y$ as shown in Figure \ref{Fig100} (iv) we get another top dimensional pants decomposition, say $W$, such that $W = (P \setminus \{b\}) \cup \{y\}$. We see that $y$ only has small intersection with $b$, and $y$ is disjoint from all the other curves in $P$, and $a$ and $c$ are not adjacent w.r.t. $W$. But if $a', b', c'$ are as shown in Figure \ref{Fig102} (i) or \ref{Fig102} (ii), $a'$ and $c'$ have to be adjacent to each other w.r.t. $W'$ that gives a contradiction. Hence, $a'$ is adjacent to $b'$. w.r.t. $P'$.

Suppose there is a curve in $P$ on the side of $c$ which doesn't contain $a$, and the side of $b$ which doesn't contain $a$ doesn't have
any curves from $P$. Then, by using nonadjacency is preserved we see that $c'$ is a separating curve, $a', b'$ are on the same side of $c'$,
and there are not any other curves of $P'$ on the side of $c'$ containing $a', b'$. This implies that $a'$ and $b'$ are adjacent w.r.t. $P'$.

Suppose the sides of each of $b$ and $c$ which doesn't contain $a$, don't have any essential curves in $P$. Then we have $(g, n)= (1, 4)$,
and $a, b, c$ are as shown in Figure \ref{Fig100} (v). In this case, we know that $a'$ is 1-sided and $b', c'$ are separating curves by Lemma \ref{1-sd}. Suppose $a'$ and $b'$ are not adjacent w.r.t. $P'$. Then, $a', b', c'$ are as shown Figure \ref{Fig102} (iii). In such a case by considering curves $z, y$ given in Figure \ref{Fig100} (v) we get a contradiction as follows: $a, c, y, z$ are pairwise nonisotopic and each of $y, z$ has small intersection with $a$, and they are disjoint from $c$. Let $y', z'$ be representatives of $\lambda([y]), \lambda([z])$ respectively such that $a', c', y', z'$ have minimal intersection. Each of $a', c', y', z'$ are pairwise nonisotopic and each of $y', z'$ intersects $a'$, and is disjoint from $c'$. This gives a contradiction. So, $a'$ is adjacent to $b'$. w.r.t. $P'$.\end{proof}\\

By Lemma \ref{tp}, the curves in a pants decomposition that corresponds to a top dimensional simplex in $\mathcal{C}(N)$ are either separating or 1-sided with nonorientable complement. So,
combining our results in Lemma \ref{adjacent1}, Lemma \ref{adjacent2} and Lemma \ref{adjacent3} we get the following:

\begin{lemma}
\label{adjacent} Suppose that $(g, n)= (3, 0)$ or $g + n \geq 4$. Let $\lambda : \mathcal{C}(N) \rightarrow \mathcal{C}(N)$
be an injective simplicial map. Let $P$ be a pair of pants decomposition on $N$ which corresponds to a top dimensional maximal
simplex in $\mathcal{C}(N)$. Let $a, b \in P$ such that $a$ is adjacent to $b$ w.r.t. $P$. There exists $a'  \in \lambda([a])$
and $b' \in \lambda([b])$ such that $a'$ is adjacent to $b'$ w.r.t. $P'$ where $P'$ is a set of pairwise disjoint curves
representing $\lambda([P])$ containing $a', b'$.\end{lemma}

\begin{lemma}
\label{1-sided-cn-2} Let $g \geq 2$. Suppose that $(g, n) = (3, 0)$ or $g+n \geq 4$. Let
$\lambda : \mathcal{C}(N) \rightarrow \mathcal{C}(N)$ be an injective simplicial map. If $a$ is a 1-sided simple
closed curve on $N$ whose complement is nonorientable, then $\lambda([a])$ is the isotopy class of a 1-sided simple
closed curve whose complement is nonorientable.
\end{lemma}

\begin{proof} The proof follows as in the proof of Lemmas 3.10 given by the author in \cite{Ir6}, by using Lemma \ref{tp},
Lemma \ref{smallint}, Lemma \ref{nonadjacent}, Lemma \ref{adjacent}.\end{proof}\\

The following theorem is given by the author in \cite{Ir5}.

\begin{theorem}
\label{super} Let $N$ be a compact, connected, nonorientable surface of genus $g$ with
$n$ boundary components. Suppose that either $(g, n) \in \{(1, 0), (1, 1), (2, 0), (2, 1), (3, 0)\}$ or $g + n \geq 5$.
If $\lambda : \mathcal{C}(N) \rightarrow \mathcal{C}(N)$ is a superinjective simplicial map, then $\lambda$ is
induced by a homeomorphism $h : N \rightarrow N$.\end{theorem}

Now we state our main result:

\begin{theorem} Let $N$ be a compact, connected, nonorientable surface of genus $g$ with
$n$ boundary components. Suppose that $g + n \leq 3$ or $g + n \geq 5$. If
$\lambda : \mathcal{C}(N) \rightarrow \mathcal{C}(N)$ is an injective simplicial map,
then $\lambda$ is induced by a homeomorphism $h : N \rightarrow N$.\end{theorem}

\begin{proof} If $g + n \leq 3$, then the proof follows similar to proofs given for small genus cases
in \cite{Ir5}. For the other cases, the proof follows by following the proof of Theorem \ref{super} given in \cite{Ir5},
by using induction on $g$ and using Lemma \ref{smallint}, Lemma \ref{nonadjacent}, Lemma \ref{adjacent} and
Lemma \ref{1-sided-cn-2}.\end{proof}



eirmak@bgsu.edu

Bowling Green State University

Department of Mathematics and Statistics

Bowling Green, 43403, OH
\end{document}